\documentclass{amsart}
\usepackage[utf8]{inputenc}
\usepackage{amssymb}
\usepackage{graphicx}
\usepackage{amsmath}
\usepackage[utf8]{inputenc}
\usepackage[english]{babel}
\usepackage{mathptmx}
\usepackage[11pt]{moresize}
\usepackage{amsthm}
\newcommand\numberthis{\addtocounter{equation}{1}\tag{\theequation}}
\newtheorem{theorem}{Theorem}[section]

\newtheorem{lemma}{Lemma}[section]

\title{Integral Ricci curvature bounds for possibly collapsed spaces with Ricci curvature bounded from below }
\author{Michael Smith }
\date{}
\email{doctorq@berkeley.edu}
\address{UC Berkeley, Berkeley, CA 94720}
\begin{document}
\begin{abstract}
Assuming a lower bound on the Ricci curvature of a complete Riemannian manifold, for $p< 1/2$ we show the existence of bounds on the local $L^p$ norm of the Ricci curvature that depend only on the dimension and which improve with volume collapse. 
\end{abstract}
\maketitle

\section{Introduction}
The main result of this paper is the following:
\begin{theorem}
Let $(M^n,g)$ be a complete Riemannian manifold and $p_0\in M$. Assume the Ricci curvature satisfies $Ric\geq-1$ on the ball $B(p_0,5)$. Then for any $0<\epsilon<1/2$ there is a $C=C(\epsilon,n)$ such that
$$||Ric||^{\frac{1}{2}-\epsilon}_{L^{1/2-\epsilon}(B(p_0,1))}\leq CVol(B(p_0,1))^{2\epsilon}$$
\end{theorem}

Where the norm is defined as $$||f||_{L^{1/2-\epsilon}(B(p_0,1))}:=\Bigg(\int_{B_R(p_0,1)}(f)^{\frac{1}{2}-\epsilon}dvol\Bigg)^{\frac{1}{1/2-\epsilon}}$$ 
Note that the bound does not deteriorate as volume collapses. Petrunin \cite{petrunin} obtains an integral bound for the scalar curvature assuming a lower bound on the sectional curvature which also holds in the collapsed setting, and Cheeger and Naber \cite{cheegernaber} obtain integral bounds for the full Riemann curvature tensor for $p<2$ assuming both upper and lower bounds on the Ricci curvature, but with the assumption that the volume is sufficiently noncollapsed. 

The idea of the proof is the following: we show that we can achieve a similar bound to that in theorem 1.1 for the Ricci curvature along a geodesic. i.e.
\begin{lemma}
Let $M^n$ a Riemannian manifold. Let $\gamma:[0,l]\rightarrow M$ a minimizing geodesic parametrized by arc-length with $l\leq2$, and assume that $Ric\geq-1$ on the image of $\gamma$. Then
$$\int_0^l|Ric(\gamma'(t),\gamma'(t))|^{\frac{1}{2}-\epsilon}dt\leq C(\epsilon)$$
\end{lemma}

With lemma 1.1 in mind, we define a class of sufficiently long geodesics and show, in some sense to be clarified after the notation is fixed (proposition 1.3 below), that if the volume of $B(p_0,1)$ is small, there cannot be too many of these geodesics. Conversely (proposition 1.4) we show (with an added diameter assumption which will not play a significant role in our main goal) that if the volume is not too small then there is a good amount of these geodesics. We use these two relationships to deduce theorem 1.1 assuming lemma 1.1. The proof of lemma 1.1 is saved for after the proof of theorem 1.1.

\section{Notation and Preliminaries}

    Let $(M^n,g)$ be an $n$-dimensional complete Riemannian manifold and $TM$ its tangent bundle. Let $SM$ be the sub-bundle of unit length vectors. We use the notation $u=(p,v)\in TM$ (or $SM$), where $p\in M$ and $v\in T_pM$, to identify an arbitrary point in $TM$. Furthermore,  we take for granted the inclusion $ T_pM\subset TM$, so that we can for example take $U\in TM$, define $V:=U\cap T_pM$ and think of $V$ as a set of vectors in the tangent space at $p$. Ignoring this detail greatly simplifies notation. There is an obvious identification $T_{(p,v)}TM=T_pM \oplus T_pM$ and we let $X$ be the vector field on $TM$ such that $X(p,v)=(v,0)$. Let $\phi:\mathbb{R}\times TM\rightarrow TM$ the geodesic flow on $TM$ and $\phi|_t$ the flow restricted to time $t$, i.e.
$$\frac{d}{dt}\phi(t_0,p)=X(\phi(t_0,p))$$
and
$$\phi|_t(p)=\phi(t,p)$$ 
\\
It is a fact that, with respect to the Liouville measure, the geodesic flow preserves volume over $SM$. Therefore, for an open set $U\subset SM$ and $t\in \mathbb{R}$, we can express the integral of a function $f:SM\rightarrow \mathbb{R}$ over the subset $\phi|_t(U)$ as follows:
\begin{align*}
\int_{\phi|_t(U)}f(u)du &= \int_{U}(\phi|_t^* f)(u)du\\
&= \int_{\pi(U)}\int_{S_pM\cap U}f(\phi|_t(u))d_pu \ dp \numberthis \label{eqn1}
\end{align*}
where $du$ is the Liouville measure on $SM$, $\pi:SM\rightarrow M$ is the usual projection, $dp$ expresses the volume form in $M$, and $d_pu$ is the Lebesgue measure on $S_pM$, i.e. it locally satisfies $d_pu\times dp=du$. 
Let 
$$seg(p):=\{v\in T_pM| \ exp_p(tv) :[0,1] \rightarrow M \ is \ minimizing\}$$
be the segment domain and 
$$seg_0(p):=\{tv| \ 0<t<1, v\in seg(p)\}$$
 Define $$W_{t,p}:=\{v \in S_pM| \ tv\in seg_0(p) \}$$  This is the collection of unit vectors based at $p$ whose corresponding geodesics, when parameterized by unit speed, are minimizing on the interval $[0,t]$.

To avoid complicating notation, we fix an arbitrary $p_0\in M$ now and define  
\begin{align*}
    W_t:=\underset{p\in B(p_0,3)}{\bigcup} W_{t,p} \numberthis  \label{def1}
\end{align*}
 Where, as mentioned above, we take for granted the inclusion $W_{t,p}\subset SM$. $W_t$ can of course be defined for any $p_0\in M$ this way, so when this notation appears below in the proofs, we understand that it has been defined for the appropriate, arbitrarily chosen $p_0$ occurring there. 
Observe that 

\begin{align*}
\int_{\phi|_t(W_t)}f(u)du &= \int_{W_t}(\phi|_t^*f)(u)du\\
&= \int_{B(p_0,3)}\int_{W_{t,p}}f(\gamma_u(t)) \ d_pu \ dp \numberthis \label{eqn2}
\end{align*}

for $u=(p,v)\in SM$, let $$R(u):=|Ric(v,v)|$$ Note that $R$ is a function defined on the unit tangent bundle $SM$, and does not denote, as $R$ often does, the scalar curvature on $M$. Let $Ric_p$ denote the Ricci tensor at $p\in M$ and $|Ric|(p)$ denote its magnitude, i.e.
$$|Ric|(p):=<Ric_p,Ric_p>^{\frac{1}{2}}$$
 For $u\in SM$, let $\beta(u)$ denote the distance to the cut locus in the direction of $u$, i.e.
 $$\beta(u):=\sup\{t>0| tv\in seg(p)\}$$ For $k>0$ let $\beta_k(u):=min\{\beta(u),k \}$.

For every $p\in M$, let  $v_p\in S_pM$ be an eigenvector corresponding to the largest (in absolute value) eigenvalue of $Ric_p$.
Keep in mind the relation
$$|Ric_p(v_p,v_p)|\geq\frac{1}{n}|Ric|(p)$$
Now let $\mu>0$ be a small constant whose exact value will be determined later.
let $$S_{p,\mu}:=\{v\in S_p M \ | |g(v,v_p)|>\mu\}$$
which we think of this set as those unit vectors which point "roughly" along the same line as $v_p$. For $v\in S_{p,\mu}$,
$$|Ric_p(v,v)|\geq \frac{\mu^2}{n}|Ric|(p)$$
Observe that for all $p\in M$
\begin{equation}
d_pu(S_pM\backslash S_{p,\mu})\leq C_1\mu \numberthis \label{eq7}
\end{equation}
for some constant $C_1$ depending only on the dimension. Finally, let $S_{\mu}$ be the union of $S_{p,\mu}$ over all $p\in M$. \\
With the notation clarified, we now state the relationships between sufficiently long minimizing geodesics and volume mentioned in the introduction. We save their proofs for after the proof of theorem 1.1. It may be useful throughout to keep in mind that, by the Bishop Gromov volume comparison [3 ch. 9, lemma 36], for any $r_1, r_2>0$, a term $C_1Vol(B(p_0, r_1)$ appearing in an inequality can always be replaced by a term $C_2Vol(B(p_0,r_2))$ for suitable $C_2$ depending on the ratio $r_1/r_2$ (In particular it is more or less inconsequential that the radii in propositions 1.3 and 1.4 are different, so long as their ratio is bounded away from zero and infinity). 

\begin{lemma}
let $(M^n,g)$ be a Riemannian manifold. Assume $p_0\in M$ is given and the set $W_1$ is defined as in \eqref{def1} in the preliminaries, then
$$Vol(W_1)\leq C_1(n)Vol(B(p_0,5))^2$$
for some $C_1(n)$ depending only on the dimension.

\end{lemma}
\begin{lemma} Let $(M^n,g)$ be a Riemannian manifold. Assume $p_0\in M$ is given and the set $W_1$ is defined as in \eqref{def1} in the preliminaries and
assume the diameter $diam(M,g)$ of $M$ satisfies $diam(M,g)\geq6$. Then there exist $\mu(V), \delta(V)>0$, depending on $V:=vol(B(p_0,2))$ and $n$, such that for all $p\in B(p_0,2)$, 
$$d_pu(\phi|_t(W_t) \cap S_{p,\mu} )>\delta(V)$$
for all $1<t<2$. Furthermore, 
$$\mu(V)\geq c_2(n)V$$
$$\delta(V)\geq c_3(n)V$$
for $c_1(n), c_2(n)$ depending only on the dimension.
\\
(we will later ignore the dependence on $V$ and write $\delta=\delta(V)$ etc.)
\end{lemma}

\section{Proof of theorem 1.1 assuming lemmas 1.1, 2.1, 2.2}

\begin{proof}
we first argue why we can assume $diam(M,g)\geq6$. Assume $R:=\frac{6}{diam(M,g)}>1$. Define the rescaled manifold $(M,g_R)$, where $g_R(\cdot,\cdot):=R^2g(\cdot,\cdot)$. Let $Ric_R$, $B_R(p_0,r)$ and $d_Rp$ denote the Ricci curvature, $r$-ball and volume form respectively, each with respect to $g_R$. Note that, because $R>1$, we preserve the condition $Ric\geq-1$ for the rescaled manifold.  

\begin{align*}
   ||Ric_R||^{\frac{1}{2}-\epsilon}_{L^{1/2-\epsilon}(B_R(p_0,R))}&=\int_{B_R(p_0,R)}(|Ric_R|(p))^{\frac{1}{2}-\epsilon}d_Rp\\
    &=\int_{B(p_0,1)}(R^{-2})^{\frac{1}{2}-\epsilon}(|Ric|(p))^{\frac{1}{2}-\epsilon}R^ndp\\
    &=R^{n-1+2\epsilon}||Ric||^{\frac{1}{2}-\epsilon}_{L^{1/2-\epsilon}(B(p_0,1))}
\end{align*}
so that
$$||Ric||^{\frac{1}{2}-\epsilon}_{L^{1/2-\epsilon}(B(p_0,1))}\leq R^{-n+1-2\epsilon} ||Ric_R||^{\frac{1}{2}-\epsilon}_{L^{1/2-\epsilon}(B_R(p_0,R))}$$
Now, we can cover $B_R(p_0,6)$ with $K$-many balls $B_R(p_i,1)$, where $K$ only depends on the dimension and a lower bound on the Ricci curvature, and so assuming we have proven the theorem for the case $diam(M,g)\geq6$, and using the Bishop-Gromov volume comparison,  it follows that 
\begin{align*}
||Ric||^{\frac{1}{2}-\epsilon}_{L^{1/2-\epsilon}(B(p_0,1))}&\leq  R^{-n+1-2\epsilon} \cdot \sum||Ric_R||^{\frac{1}{2}-\epsilon}_{L^{1/2-\epsilon}(B_R(p_i,1))}\\
&\leq R^{-n+1-2\epsilon} \cdot KC(n)\cdot Vol_R(B_R(p_i,1))^{2\epsilon}\\
&\leq  R^{-n+1-2\epsilon} \cdot C'(n)\cdot Vol_R(B_R(p_0,1))^{2\epsilon}\\
&\leq R^{-n+1-2\epsilon} \cdot C''(n)\cdot Vol_R(B_R(p_0,R))^{2\epsilon}\\
&= R^{-n+1-2\epsilon} \cdot C''(n)\cdot R^{2\epsilon n}Vol(B(p_0,1))^{2\epsilon}\\
&\leq  C''(n)Vol(B(p_0,1))^{2\epsilon}
\end{align*}
where the last line follows because $R>1$, $n\geq2$ and $\epsilon<1/2$.
Now, with the assumption that the diameter of $M$ satisfies $diam(M,g)\geq6$, we assume lemmas 1.1, 2.1, and 2.2 hold. 
Firstly, for $1<t<2$ we have 

 \begin{align*}
 \int_{B(p_0,3)}\int_{W_{t,p}}R(\phi|_t(u))^{\frac{1}{2}-\epsilon}d_pu \ dp &=\int_{W_t}R(\phi|_t(u))^{\frac{1}{2}-\epsilon}du\\
&= \int_{\phi|_t(W_t)}R(u)^{\frac{1}{2}-\epsilon}du \\
&\geq \int_{\phi|_t(W_t) \bigcap S_{\mu}}R(u)^{\frac{1}{2}-\epsilon}du\\
&=\int_{M}\int_{\phi|_t(W_t) \cap S_{p,\mu}}R(p,v)^{\frac{1}{2}-\epsilon}d_pu \ dp \\
&\geq \int_{B(p_0,1)}\int_{\phi|_t(W_t) \cap S_{p,\mu}}R(p,v)^{\frac{1}{2}-\epsilon}d_pu \ dp \\
&\geq \int_{B(p_0,1)}\int_{\phi|_t(W_t) \cap S_{p,\mu}}(\frac{\mu^2}{n}|Ric|(p))^{\frac{1}{2}-\epsilon}d_pu \ dp \\
&\geq \int_{B(p_0,1)}\delta(\frac{\mu^2}{n})^{\frac{1}{2}-\epsilon}(|Ric|(p))^{\frac{1}{2}-\epsilon}dp \\
 &\geq  \delta(\frac{\mu^2}{n})^{\frac{1}{2}-\epsilon}||Ric||^{\frac{1}{2}-\epsilon}_{L^{1/2-\epsilon}(B(p_0,1))}\\
 &\geq c Vol(B(p_0,1))^{2-2\epsilon} ||Ric||^{\frac{1}{2}-\epsilon}_{L^{1/2-\epsilon}(B(p_0,1))} \numberthis
\end{align*}
 where we have used $\mu=\mu(V)$ and $\delta=\delta(V)$ as in lemma 2.2 to determine $c>0$.

Because this holds for all $1<t<2$, we similarly obtain
\begin{align*}
 \int_1^2\int_{B(p_0,3)}\int_{W_{t,p}}R(\phi|_t(u))d_pu \ dt \ dp &\geq \int_1^2cVol(B(p_0,1))^{2-2\epsilon}||Ric||^{\frac{1}{2}-\epsilon}_{L^{1/2-\epsilon}(B(p_0,1))} dt \numberthis \label{eqn5}\\
&\geq c Vol(B(p_0,1))^{2-2\epsilon}||Ric||^{\frac{1}{2}-\epsilon}_{L^{1/2-\epsilon}(B(p_0,1))}
\end{align*}

We obtain an upper bound for the term on the left as follows. By Fubini's theorem, lemmas 1.1 and 2.1 and the assumption that $Ric\geq -1$ on $B(p_0,5)$,
\begin{align*}
\int_1^2\int _{B(p_0,3)}\int_{W_{t,p}} R(\phi|_t(u))^{\frac{1}{2}-\epsilon}d_pu \ dt \ dp
&=\int _{B(p_0,3)}\int_{W_{1,p}}\int_1^{\beta_2(u)}R(\phi|_t(u))^{\frac{1}{2}-\epsilon}dt \ d_pu \ dp\\
&\leq \int _{B(p_0,3)}\int_{W_{1,p}} C(\epsilon) \ \ d_pu \ dp \\
&\leq C(\epsilon)Vol(W_1) \\
&\leq C(\epsilon)Vol(B(p_0,5))^2\\
&\leq C'(\epsilon)Vol(B(p_0,1))^2
\end{align*}

It finally follows that
\begin{align*}
 c Vol(B(p_0,1))^{2-2\epsilon}||Ric||^{\frac{1}{2}-\epsilon}_{L^{1/2-\epsilon}(B(p_0,1))}
&\leq C'(\epsilon)Vol(B(p_0,1))^2
\end{align*}
which is theorem 1.1.
\end{proof}

\section{Proofs of lemmas 1.1, 2.1, 2.2}

\begin{proof}[Proof of lemma 1.1]
Along a geodesic $\gamma$ parametrized by arc length, choose $E_1(t),...,E_{n-1}(t)$ to be orthonormal parallel vector fields along $\gamma$ that are perpendicular to $\gamma'$. If $h$ is a nonnegative continuously differentiable function, by the second variation formula, 
$$0\leq \int_0^l (h'(t))^2dt-\int_0^l(h^2 sec(\gamma'(t),E_i(t)))dt$$
Summing over $i=1,...,n-1$ gives 
\begin{equation}
\int_0^l h(t)^2 Ric(\gamma'(t),\gamma'(t)) dt\leq (n-1)\int_0^l (h'(t))^2dt \label{2var}
\end{equation}
Let $Ric_+(\cdot,\cdot):=max\{Ric(\cdot,\cdot),0\}$ and $Ric_-(\cdot,\cdot):=-min\{Ric(\cdot,\cdot),0\}$. By assumption $Ric_-\leq1$. Therefore, we only need to show 

$$\int_0^lRic_+(\gamma'(t),\gamma'(t))^{\frac{1}{2}-\epsilon}dt\leq C(\epsilon)$$
to complete the proof.
It follows from \eqref{2var} that
\begin{align*}
   0&\leq \int_0^lh(t)^2 Ric_+(\gamma'(t),\gamma'(t)) dt \\
   &\leq (n-1)\int_0^l (h'(t))^2dt+\int_0^lRic_-(\gamma'(t),\gamma'(t)) dt\\
   &\leq (n-1)\int_0^l (h'(t))^2dt+2 \numberthis \label{bound}\\
\end{align*}
 since $l\leq2$. Define $h_l(t):=t$ on $[0,l/2]$, and $h_l(t):=(l/2-t)$ on $[l/2,l]$, modified in $[l/4,3l/4]$ so that the resulting function is smooth on $[0,l]$ and preserving both that $|h_l|>l/4$ and $|h_l'|\leq1$ hold within $[l/4,3l/4]$. For fixed $\mu>0$, using \eqref{bound} and the construction of $h_l$

\begin{align*}
\int_0^l h_l(t)^{1+\mu}Ric_+(\gamma'(t),\gamma'(t))dt &=
\int_0^l \Big(h_l(t)^{\frac{1+\mu}{2}}\Big)^2 Ric_+(\gamma'(t),\gamma'(t)) dt\\
&\leq (n-1)\int_0^l \Big((\frac{1+\mu}{2})h_l^{-\frac{1+\mu}{2}}h_l'(t)\Big)^2dt +2\\
&= (n-1)\int_0^{l/4} \Big((\frac{1+\mu}{2})h_l^{-\frac{1+\mu}{2}}h_l'(t)\Big)^2dt\\
&+(n-1)\int_{l/4}^{3l/4} \Big((\frac{1+\mu}{2})h_l^{-\frac{1+\mu}{2}}h_l'(t)\Big)^2dt+2\\
&+(n-1)\int_{3l/4}^l \Big((\frac{1+\mu}{2})h_l^{-\frac{1+\mu}{2}}h_l'(t)\Big)^2dt\\
&\leq 2(n-1)\frac{(1+\mu)^2}{4}\int_0^{l/4}t^{-1+\mu}dt+C\\
&< C(\epsilon)
\end{align*}

Let $0<\epsilon<1/2$ be given. Using Hölder's inequality with exponents $p=\frac{2}{1-2\epsilon}$ and $q=\frac{2}{1+2\epsilon}$ gives
\begin{align*}
    \int_0^l Ric_+(\gamma'(t),\gamma'(t))^{\frac{1}{2}-\epsilon}&= \int_0^l \frac{h_l(t)^{\frac{1-\epsilon}{2}}}{h_l(t)^{\frac{1-\epsilon}{2}}}Ric_+(\gamma'(t),\gamma'(t))^{\frac{1}{2}-\epsilon}dt \\
    &\leq \Bigg(\int_0^l h_l(t)^{\frac{1-\epsilon}{1-2\epsilon}}Ric_+(\gamma'(t),\gamma'(t))dt\Bigg)^{1/p}\Bigg( \int_0^lh_l(t)^{\frac{\epsilon-1}{1+2\epsilon}}dt\Bigg)^{1/q} \\
    &=\Bigg(\int_0^l h_l(t)^{1+\mu}Ric_+(\gamma'(t),\gamma'(t))dt\Bigg)^{1/p}\Bigg( \int_0^lh_l(t)^{\frac{\epsilon-1}{1+2\epsilon}}dt\Bigg)^{1/q}\\
    &\leq C(\epsilon)\Bigg( \int_0^lh_l(t)^{\frac{\epsilon-1}{1+2\epsilon}}dt\Bigg)^{1/q}\\
    &\leq C(\epsilon)\Bigg( \int_0^lt^{\frac{\epsilon-1}{1+2\epsilon}}dt\Bigg)^{1/q}\\
    &\leq C'(\epsilon)
    \end{align*}
    Where we have chosen $\mu$ to satisfy $1+\mu=\frac{1-\epsilon}{1-2\epsilon}$.

\end{proof}

Fix $u=(p,v)\in SM$. $u$ determines a minimizing geodesic $\gamma_v:[0,l]\rightarrow M$ parameterized by unit speed for some $l>0$. In exponential coordinates based at $p=\gamma_v(0)$, the volume form can be expressed as 
$$f(t,v)dt\wedge dv$$
where $t$ represents the radial coordinate and $v\in SM$ determines a direction in $S^{n-1}$. Restricting to our fixed $v$, this determines a function $f_v(t)$ for $0<t<l$, which we think of as the magnitude of the volume form along $\gamma$ starting from $p$. By volume comparison and our curvature assumption, this function cannot be larger than the corresponding function for a manifold of constant curvature equal to -1. It can, however, be much smaller.
\\
\ If we consider instead a new starting point along $\gamma_v$, say $\gamma_v(s)$ for some $s<l$, we can let $\bar{v}:=\gamma_v'(s)$ and consider the magnitude of the volume form in exponential coordinates along $\gamma_{\bar{v}}$ based at $\gamma_{\bar{v}}(0)=\gamma_v(s)$. This defines a new function $f_{\bar{v}}(t)$ for $0<t<l-s$ in the same fashion as above. In this way, we can examine the magnitude of the volume form along any minimizing geodesic starting from any base point along that geodesic. Accordingly, we define $F:SM\times \mathbb{R}_+\rightarrow \mathbb{R}$ so that it satisfies, for a given $u=(p,v)\in SM$ and $t\in (0,\beta(u))$
$$F(u,t):=f_v(t)$$

We will use the following, lemma 9 from \cite{croke}, which says that on average, the function determined in this way cannot be too small. 

\begin{lemma}
Let $(M^n,g)$ be a complete Riemannian manifold and $u\in SM$. Then for every $l\leq \beta(u)$ (the distance to the cut locus in the direction of $u$):
$$\int_{0}^{l}\int_{0}^{l-t}F(\phi|_t(u),s)ds\ dt\geq C(n)\frac{l^{n+1}}{\pi^{n+1}}$$
Where $F(u,s)$ restricted to $S_pM\times\mathbb{R}_+$ is the magnitude of the volume form in exponential coordinates on $M$, i.e. it satisfies
\begin{align*} 
\int_{S_pM}\int_0^{\beta(u)}F(u,s)ds \ dv=Vol(M) \numberthis \label{def3}
\end{align*}

\end{lemma}

\begin{proof}[Proof of lemma 2.1] 

The following inequality follows directly from the lemma, 
\begin{align*}
  &\int_{0}^{1}\int_{B(p_0,3)}\int_{W_{1,p}} \int_{0}^{1-t} F(\phi|_t(u),s) ds \ d_pu  \ dp \ dt \\
  &= \int_{B(p_0,3)}\int_{W_{1,p}} \int_{0}^{1}\int_{0}^{1-t} F(\phi|_t(u),s) ds \ dt\ d_pu  \ dp \\
  &\geq \int_{B(p_0,3)}\int_{W_{1,p}}c  \ d_pu \ dp \\
  &\geq cVol(W_1)
\end{align*}
We also have 
\begin{align*}
    \int_{B(p_0,4)}\int_{S_pM}\int_{0}^{\beta_1(u)} F(u,s) ds \ d_pu  \ dp  &\leq Vol(B(p_0,5)^2 \\
\end{align*}
which is clear from the definition of $F$ and $\beta_1$, and from \eqref{def3}.
It then follows that
\begin{align*}
    &\int_{0}^{1}\int_{B(p_0,3)}\int_{W_{1,p}} \int_{0}^{1-t} F(\phi|_t(u),s) ds \ d_pu  \ dp \ dt \\
    &=\int_{0}^{1}\int_{W_1} \int_{0}^{1-t} F(\phi|_t(u),s) ds \ du \ dt \\
    &= \int_{0}^{1}\int_{\phi|_t(W_1)} \int_{0}^{1-t} F(u,s) ds \ du \ dt \\
    &\leq  \int_{0}^{1}\int_{B(p_0,4)}\int_{S_pM} \int_{0}^{\beta_1(u)} F(u,s) ds \ d_pu  \ dp \ dt \\
    &\leq CVol(B(p_0,5))^2\\
\end{align*}
The second to last line follows since all elements $u=(p,u_p)\in W_1$ correspond to vectors $v$ of magnitude $|v|=1$ based at some $p\in B(p_0,3)$. Therefore, the image under the geodesic flow at time $t$ of such an element is $(q,u_q)$ for some $u_q$ with $|u_q|=1$ based at $q$ for some $q\in B(p_0,4)$ if $t<1$. 

The two inequalities give the desired result.
\end{proof}

\begin{proof}[Proof of lemma 2.2]
Notice that for all $p\in B(p_0,2)$ and $v \in W_{t,p}$, for $1<t<2$,  $(p,-v)\in \phi|_t(W_t) \cap S_pM$, which is clear by reparameterizing the relevant geodesics to go in the opposite direction. Therefore,

$$d_pu (W_{t,p})\geq 2\delta $$
implies
$$d_pu(\phi|_t(W_t)\cap S_pM)\geq 2\delta$$
Therefore, if we let $\mu:=\frac{\delta}{C_1}$ with $C_1$ as in \eqref{eq7},
then
$$d_pu(\phi|_t(W_t) \cap S_{p,\mu} )=d_pu(\phi|_t(W_t)\cap S_pM)-d_pu(\phi|_t(W_t)\cap (S_pM\backslash S_{p,\mu}))\geq 2\delta-\delta \geq \delta$$
Therefore we only need to show 
$$d_pu (W_{t,p})\geq \delta $$ with $\delta_1$ on the order of $Vol(B(p_0,2))$ to complete the proof. Furthermore, observe that if we show the result for $t=2$, it immediately follows for $t<2$. \\
For any $p\in B(p_0,5)$
$$Vol(B(p,10))\geq vol(B(p_0,5))$$ 
and therefore, by Bishop-Gromov volume comparison,
$$ Vol(B(p,1))\geq C(n)vol(B(p_0,2))$$
By the diameter assumption on $M$, for any $p\in B(p_0,2)$ there exists $q\in B(p_0,5)$ such that $d(p,q)=3$. We must similarly have 
$$Vol(B(q,1))\geq C(n)Vol(B(p,4) \geq  C(n)vol(B(p_0,2))$$
But, once more using Bishop-Gromov and letting $V^n_{-1}$ denote the volume of the sphere of radius 1 in $n$-dimensional hyperbolic space, this implies
\begin{align*}
    \ 2\cdot 4^n V^n_{-1}vol(W_{2,p}) &\geq \int_{2}^{\beta_2(x)}\int_{\frac{1}{2}W_{2,p}} F(u,s)d_pu \ ds\\
    &\geq Vol(B(q,1)) \\
\end{align*}
implying
$$W_{2,p}\geq CVol(B(p_0,2))=:2\delta$$
where $F$ is defined in the proof of lemma 2.1.

\end{proof}


\begin{thebibliography}{9}
 
 
 \bibitem{petrunin} 
Anton Petrunin.
\textit{An Upper Bound for Curvature Integral}.
Algebra i Analiz, 20:2 (2008), 134–148; St. Petersburg Math. J., 20:2 (2009), 255–265 

 \bibitem{cheegernaber} 
Jeff Cheeger, Aaron Naber.
\textit{Regularity of Einstein manifolds and the codimension conjecture 4 conjecture}.
Annals of Mathematics. (2015). 182, 3, p. 1093-1165 73 p. 
 
 \bibitem{petersen} 
Peter Petersen.
\textit{Riemannian Geometry}.
Springer, Graduate Texts in Mathematics, 2 Ed. 2006.

\bibitem{croke} 
Christopher B. Croke.
\textit{Some isoperimetric inequalities and eigenvalue estimates}.
Ann. Sci. Ecole Norm. Sup.(4) 13 (1980), 419-435

\end{thebibliography}
\end{document}